\documentclass[envcountsame]{amsart}
\usepackage{amsmath,amssymb,latexsym, graphicx, mathrsfs, manfnt, textcomp, appendix, chemarrow, enumerate, enumitem, multicol, cite, soul}
\usepackage[misc,geometry]{ifsym}
\usepackage[colorlinks=true,linkcolor=blue,citecolor=blue]{hyperref}
\usepackage[all]{xy}
\usepackage[utf8]{inputenc}
\usepackage[T1]{fontenc}

\numberwithin{equation}{section}

\newcommand\ces{\mathsf{C}}
\newcommand\ices{\mathsf{C}^{-1}}
\newcommand\cn{\mathbb{C}}
\newcommand\rn{\mathbb{R}}
\newcommand\cnn{\mathbb{C}^\mathbb{N}}
\newcommand\nn{\mathbb{N}}
\newcommand\kotz{\lambda_0(A)}

\newcommand\ptsp{\sigma_{pt}(\mathsf{C};\lambda_0(A))}
\newcommand\spec{\sigma(\mathsf{C};\lambda_0(A))}
\newcommand\clo{\mathcal{L}}

\newcommand\proj{\operatorname{proj}}
\newcommand\igoes{\xrightarrow[i]{}}
\newcommand\sumi{\sum_{i=1}^\infty}
\newcommand\limi{\lim_{i \to \infty}}
\newcommand\supi{\sup_{i \in \nn}}
\newcommand\sumj{\sum_{j=1}^\infty}

\theoremstyle{thmit} % Numbered and Italic
\newtheorem{theorem}{Theorem}
\newtheorem{lemma}{Lemma}
\newtheorem{corollary}{Corollary}
\newtheorem{proposition}{Proposition}

\newtheorem{example}{Example}
\newtheorem{remark}{Remark}

\begin{document}

\author{Ersin Kızgut}
\address{Universitat Politècnica de València \\ Instituto Universitario de Matemática Pura y Aplicada (IUMPA)\\ E-46071 Valencia, Spain}
\email{erkiz@upv.es}

\title{The Cesàro operator on smooth sequence spaces of finite type}

\date{\today}


\thanks{This article has been accepted in Rev. R. Acad. Cienc. Exactas Fís. Nat. Ser. A Math. RACSAM}

\maketitle

\begin{abstract}
	The discrete Cesàro operator $\mathsf{C}$ is investigated in the class of smooth sequence spaces $\lambda_0(A)$ of finite type. This class contains properly the power series spaces of finite type. Of main interest is its spectrum, which is distinctly different in the cases when $\lambda_0(A)$ is nuclear and when it is not. The nuclearity of $\lambda_0(A)$ is characterized via certain properties of the spectrum of $\mathsf{C}$. Moreover, $\mathsf{C}$ is always power bounded and uniformly mean ergodic on $\lambda_0(A)$.
\end{abstract}

\section{Introduction}
	
The discrete Cesàro operator $\ces$ defined on $\cnn$ is defined by
\[
	\ces x:=\left(x_1, \frac{x_1+x_2}{2}, \frac{x_1+x_2+x_3}{3}, \dots, \frac{x_1+\dots +x_i}{i}, \dots \right), \quad x=(x_i)_{i \in \nn}.
\]
The Cesàro operator $\ces$ has been investigated on many Banach sequence spaces. We refer the reader to the introduction of \cite{ABRwc0}. The behaviour of $\ces$ when acting on the Fréchet spaces $\cnn$, $\ell_{p+}=\bigcap_{q>p}\ell_q,\, 1 \leq p<\infty$, and on the power series space $\Lambda_0(\alpha)$ of finite type was studied in \cite{ABR13, ABR17, ABR18}. In this paper we extend the results of \cite{ABR18} to the strictly more general setting of the smooth sequence spaces $\kotz$ of finite type. These spaces were introduced by Terzioğlu \cite{Ter69, Ter73, Ter75, Ram79}. We also refer to Kocatepe \cite{Koc85, Koc88, Koc89}. Some of our proofs are inspired by \cite{ABR18}, but new ingredients are needed in our setting. We describe precisely our context. Let $A=(a_n)_n$, where $a_n=(a_n(i))_i$. $A$ is called a \textit{Köthe matrix} if the following conditions are satisfied:
\begin{enumerate}[label=(K\arabic*)]
	\item $0 \leq a_n(i) \leq a_{n+1}(i)$, for all $i,n \in \nn$.
	\item For all $i \in \nn$, there exists $n \in \nn$ such that $a_n(i)>0$.
\end{enumerate} 
The \textit{Köthe echelon space of order} 0 associated to $A$ is defined by
\[
\kotz=\{x \in \cnn:\limi a_n(i) \, x_i=0, \, \forall n \in \nn\},
\]
which is a Fréchet space when equipped with the increasing system of seminorms
\[
	p_n(x):=\supi  a_n(i)|x_i|, \quad x \in \kotz,\, n \in \nn.
\]
Then $\kotz=\bigcap_{n \in \nn} c_0(a_n)$, with $c_0(a_n)$ the usual Banach space. We consider the projective limit topology in $\kotz$, that is, $\kotz=\proj_{n \in \nn} c_0(a_n)$. For further reading in the theory of Köthe echelon spaces $\lambda_p(A),\, 1 \leq p \leq \infty$ or $p=0$, see \cite[Section 27]{Vog97}.

The space $\kotz$ is said to be a \textit{smooth sequence space of finite type} (or a $G_1$-space) \cite[Section 3]{Ter69} if $A$ satisfies
\begin{enumerate}[label=(G1-\arabic*)]
	\item $0 < a_n(i+1) \leq a_n(i)$, for all $n \in \nn$ and $i \in \nn$.
	\item For all $n \in \nn$ there exist $m>n$ and $C>0$ such that $a_n(i) \leq C {a_m(i)}^2$, for all $i \in \nn$.
\end{enumerate}
Condition (G1-2) is actually equivalent to say that $\kotz$ is a \textit{Fréchet algebra} with respect to pointwise multiplication. Our first result is probably known. We include it for the sake of completeness and to explain the assumption~(\ref{sequences vanish}) below.
\begin{proposition}\label{G1 Schwartz}
	Let $A$ satisfy (G1-1). Then, either $\kotz$ is normable (and necessarily isomorphic to $c_0$) or $\limi a_n(i)=0$ for all $n \in \nn$.
\end{proposition}

\begin{proof}
	Since $A$ satisfies (G1-1), there exists $L_n \geq 0$ such that $\limi a_n(i) = L_n$, for all $n \in \nn$, and $L_n \leq a_n(i) \leq a_n(1)$, for all $i,n \in \nn$. Then we have two possibilities
	\begin{enumerate}[label=\normalfont(\arabic*)]
		\item There exists $n_0 \in \nn$ such that $L_{n_0} > 0$,
		\item For all $n \in \nn$, $L_n=0$.
	\end{enumerate}
	In case (1), for all $n \geq n_0$ and for all $i \in \nn$, $0 < L_{n_0} \leq a_{n_0}(i) \leq a_n(i)$. So $0 <L_{n_0} \leq L_n$ for $n \geq n_0$. Therefore $0 < L_{n_0} \leq a_n(i) \leq a_n(1)$, for all $i \in \nn$, for all $n \geq n_0$. The latter implies
	\[
		a_n(i)|x_i| \leq a_n(1) |x_i| = \frac{a_n(1)}{L_{n_0}} L_{n_0} |x_i| \leq \frac{a_n(1)}{L_{n_0}} a_{n_0}(i)|x_i|.
	\]
	Taking supremum on both sides for $i \in \nn$, we obtain
	\[
		p_n(x) \leq \frac{a_n(1)}{L_{n_0}} p_{n_0}(x).
	\]
	Hence $\kotz$ is normable since $p_{n_0}$ is the norm defining the topology of $\kotz$.  
\end{proof}
In the light of Proposition~\ref{G1 Schwartz}, given the $G_1$-space $\kotz$, we shall assume 
\begin{equation}\label{sequences vanish}
\limi a_n(i)=0, \quad \forall n \in \nn ,
\end{equation}
since otherwise the situation is totally clarified in \cite{ABRwc0}. Therefore, in our case, every $G_1$-space $\kotz$ is a Schwartz space, hence a Montel space \cite[3.4]{Ter69}.

A power series space $\Lambda_0(\alpha)=\{x \in \cnn | \limi \exp(-\alpha_i/n)|x_i|=0\}$ of finite type for $\alpha_i \igoes \infty$ is a $G_1$-space (see \cite[Section 29]{Vog97} for power series spaces of finite type). But the converse is false, in general as shown in Example~\ref{g1_non_pss} below. Before that, let us remind: A Fréchet space $E$ with a fundamental system $(p_n(\cdot))_n$ of seminorms is said to have the property (\ul{DN}) \cite[pp. 368]{Vog97} if there exists $k \in \nn$ so that for every $n \in \nn$ there exist $m \in \nn$, $0 < \tau <1$ and $C>0$ with
\[
p_n(x) \leq C \, p_k(x)^{1-\tau}\,p_m(x)^\tau, \quad \forall x \in E. \tag{\ul{DN}}\label{weak-DN}
\]

\begin{example}\label{g1_non_pss}
\normalfont
The space $X:=\{x \in \cnn:\limi a_n(i)x_i=0,\, \forall n \in \nn\}$, where $(a_n(i))_{i,n}:=\{\exp(-n e^{i/n}) \}$, is a nuclear $G_1$-space which is not isomorphic to a power series space of finite type.
\end{example}

\begin{proof}
\begin{enumerate}[wide, labelwidth=!, labelindent=0pt, label=(\textbf{\roman*})]
\item \textit{$X$ is a Köthe echelon space}: For all $i,n \in \nn$ satisfying $i \geq n(n+1)\log(1+\frac{1}{n})$, we have
		\begin{align*}
			& &\frac{i}{n(n+1)}  =  i \left(\frac{1}{n}-\frac{1}{n+1}\right) \geq & \log\left(1+\frac{1}{n}\right) = \log(n+1)-\log(n) \\
			&\Rightarrow &\log(n)+\frac{i}{n}  \geq & \, \frac{i}{n+1}+\log(n+1) \\
			&\Rightarrow & n e^{\frac{i}{n}}  \geq & (n+1) e^{\frac{i}{n+1}}\\
			 &\Rightarrow & \frac{1}{ne^{\frac{i}{n}}}  \leq  & \frac{1}{(n+1)e^{\frac{i}{n+1}}}.
		\end{align*}
Hence we have $a_n(i) \leq a_{n+1}(i)$, for all $i \geq n(n+1)\log(1+\frac{1}{n})$. Since only a finite number of indices for $i$ remain for each $n$, one can inductively choose $\{b_k\}$ each of which greater than 1, such that $a_k(i) \leq b_k(i) a_{k+1}(i)$ is satisfied for all $i \in \nn$. Observing $X \simeq \lambda_0(C)$ with the choice $c_n:=(b_n a_n)$ we deduce that $X$ is a Köthe echelon space of order 0.
		
\item \textit{$X$ satisfies (G1-1)}: Obviously for all $i \in \nn$, one has $\exp(-n e^{\frac{i+1}{n}}) < \exp(-n e^{\frac{i}{n}})$, so $a_n(i+1) < a_n (i)$, for all $i \in \nn$ and (G1-1) is satisfied.
		
\item \textit{$X$ satisfies (G1-2)}: Given $n$, choose $m=2n$. Then, one has $a_m(i)^2=\exp(-2me^{i/m})=\exp(-4ne^{i/{2n}})$. Thus $a_n(i)=\exp(-ne^{i/n}) \leq a_m(i)^2$ if and only if $-ne^{i/n} \leq -4n e^{i/{2n}}$ if and only if $i/n \geq \log(4)+(1/{2n})$ if and only if $i \geq n(2 \log (4))$. Since $1 < \log(4) < 2$, the latter is certainly satisfied for all $i \geq 4n$. So $a_n(i) \leq a_m(i)^2$, for all $i \geq 4n$. Because only a finite number of indices for $i$ remain, and $a_k(i)>0$ for all $i,k \in \nn$, there is $M>0$ such that $a_n(i) \leq Ma_m(i)^2$, for all $i \in \nn$.

\item \textit{$X$ is nuclear}: Observe first $\frac{\log(i)}{i} \igoes 0$. Given $n$, find $i_0 \in \nn$ such that $\frac{\log (i)}{i} \leq \frac{1}{n}$, for all $i>i_0$. Then
		\[
			\log (i)<\frac{i}{n} \quad \Rightarrow \quad i < e^{i/n} < n e^{i/n} \quad \Rightarrow \quad -n e^{i/n} < -i,
		\]
		for all $i>i_0$. Thus, $a_n(i)=\exp(-ne^{i/n})<e^{-i}$. But then, since $X$ is a $G_1$-space, for any $n \in \nn$, one may find $m>n$ and $C>0$ such that
		\[
			\sum_{i=i_0}^\infty \frac{a_n(i)}{a_m(i)} \leq \sum_{i=i_0}^\infty \frac{C\, a_m(i)^2}{a_m(i)} = \sum_{i=i_0}^\infty C\, a_m(i) \leq \sum_{i=i_0}^\infty e^{-i} < \infty.
		\]
		Hence the Grothendieck-Pietsch criterion for nuclearity (see e.g. \cite[Theorem 28.15]{Vog97}) is satisfied, so $X$ is a nuclear $G_1$-space. 
		
\item \textit{$X$ is not isomorphic to a power series space of finite type}: By \cite[Example 5-(5)]{Bra88}, this space fails the property \eqref{weak-DN}. However, by \cite[Lemma 29.12]{Vog97}, every power series space has property \eqref{weak-DN}. Therefore, $X$ cannot be isomorphic to a power series space of finite type. 
\end{enumerate}
 \end{proof}
More examples can be seen in Lemma~\ref{equivalences on Sn} and Remark~\ref{more G1}.

\section{Continuity and compactness of $\ces$ on $\kotz$}
An operator $T$ on a Fréchet space $X$ into itself is called bounded (resp. compact) if there exists a neighborhood $U$ of the origin of $X$ such that $TU$ is a bounded (resp. relatively compact) set in $X$. The following result is well-known (see e.g. \cite[Lemma 25]{ABR18-2}).

\begin{lemma}\label{proj lim cts}
	Let $E=\proj_m E_m$ and $F=\proj_n F_n$ be Fréchet spaces such that $E$ (resp. $F$) is the intersection of the sequence of Banach spaces $E_m$ (resp. $F_n$), $E$ is dense in $E_m$ and $E_{m+1} \subset E_m$ with continuous inclusion for each $m$ (resp. $F$ is dense in $F_n$ and $F_{n+1}\subset F_n$ with continuous inclusion for each $n$). Let $T\colon E \to F$ be a linear operator. Then
	\begin{enumerate}[label=\normalfont(\arabic*)]
		\item $T$ is continuous if and only if for each $n$ there is $m$ such that $T$ has a unique continuous linear extension $T_{m,n}\colon E_m \to F_n$.
		\item Assume $T$ is continuous. Then $T$ is bounded if and only if there is $m$ such that for each $n$, $T$ has a unique continuous linear extension $T_{m,n}\colon E_m \to F_n$.
	\end{enumerate}
\end{lemma}

\begin{proposition}
	Let $\kotz$ be a Köthe echelon space of order 0. Then, $\ces\colon\kotz \to \kotz$ is continuous if and only if for all $n \in \nn$ there exists $m>n$ such that
	
	\begin{equation}\label{continuity criterion}
		\left\{\frac{a_n(i)}{i} \sum_{j=1}^i \frac{1}{a_m(j)} \right\}_{i \in \nn} \in \ell_\infty.
	\end{equation}
\end{proposition}

\begin{proof}
	Follows from Lemma~\ref{proj lim cts}, and \cite[Proposition 2.2(i)]{ABRwc0}.  
\end{proof}

\begin{corollary}
	Let $A$ be a Köthe matrix satisfying the condition (G1-1). Then, $\ces\colon\kotz \to \kotz$ is continuous.
\end{corollary}

\begin{proof}
	Since $a_m(i+1) \leq a_m(i)$, for all $i \in \nn$, we find $m>n$ such that
	\[
		\frac{a_n(i)}{i} \sum_{j=1}^i \frac{1}{a_m(j)} \leq \frac{a_n(i)}{i} \frac{i}{a_m(i)} = \frac{a_n(i)}{a_m(i)} \leq 1, \quad \forall i \in \nn.
	\]
	Hence $\ces$ is continuous on $\kotz$. 
\end{proof}
The following proposition is a direct consequence of \cite[Proposition 2.2.(ii)]{ABRwc0} and Lemma~\ref{proj lim cts}.
\begin{proposition}
	Let $A$ be a Köthe matrix satisfying condition (G1-1). Then the Cesàro operator $\ces\colon\kotz \to \kotz$ is compact if and only there exists $m \in \nn$ such that for all $n$
	\begin{equation}\label{compactness criterion}
		\left\{\frac{a_n(i)}{i}\sum_{j=1}^i \frac{1}{a_m(j)}\right\}_{i \in \nn} \in c_0
	\end{equation}
	is satisfied.
\end{proposition}

Let $D\colon\cnn \to \cnn$ be the formal operator of differentiation $D(x_1, x_2, x_3, \dots) = (x_2, 2x_3, 3x_4, \dots)$, $x=(x_i)_i$.
\begin{proposition}
	 Let $A$ be a Köthe matrix satisfying condition (G1-1). Then, the following statements are equivalent:
	 \begin{enumerate}[label=\normalfont(\arabic*)]
	 	\item $D\colon\kotz \to \kotz$ is continuous.
	 	\item For all $n \in \nn$, there exist $m>n$ and $M>0$ such that
	 	\begin{equation}
	 		ia_n(i) \leq M a_m(i+1), \quad \forall i \in \nn.
	 	\end{equation}
	 \end{enumerate}
\end{proposition}

\begin{proof}
	(1) $\Rightarrow$ (2) Assume that $D$ is continuous on $\kotz$. Then for all $n \in \nn$, there exist $m>n$ and $M > 0$ such that
	\[
		\supi  a_n(i)|(Dx)_i| \leq M \supi  a_m(i) |x_i|, \quad \forall x \in \kotz.
	\]
	Letting $x=(e_j)$, the canonical basis in $\cnn$ for $j\geq 2$ we obtain
	\[
		(j-1)a_n(j-1) \leq M a_m(j), \quad \forall j \geq 2,
	\]
	which is precisely (2).
	
	(2) $\Rightarrow$ (1) Given $n$, pick $m$ and $M>0$ as in condition (2) and for $x \in \kotz$ we have
	\begin{align*}
		a_n(i)|(Dx)_i|= i a_n(i)|x_{i+1}| \leq & M a_m(i+1)|x_{i+1}| \\
										  \leq & M p_m(x),
	\end{align*}
	which implies that $D$ is continuous on $\kotz$. 
\end{proof}

\begin{example}
\normalfont
	Consider the nuclear $G_1$-space $X$ constructed in Example~\ref{g1_non_pss}. Picking $m=2n$ yields
	\begin{align*}
		\frac{ia_n(i)}{a_m(i+1)} = \frac{i\exp(-ne^{i/n})}{\exp(-2ne^{(i+1)/2n})} & = \exp(\log(i)+2ne^{1/2n}e^{i/2n}-ne^{i/n})\\
		& = \exp\left(\frac{\log(i)}{ne^{i/n}}+2e^{1/2n}e^{-i/2n}-1\right)ne^{i/n} \igoes 0.
	\end{align*}
	Hence $D\colon X \to X$ is continuous.
\end{example}

A Köthe echelon space $\kotz$ of order 0 is called \textit{regular} \cite{Dra70} if
\[
	\frac{a_n(i)}{a_{n+1}(i)} \geq \frac{a_n(i+1)}{a_{n+1}(i+1)}, \quad \forall i,n \in \nn.
\]

\begin{proposition}\label{nuclearity_criterion1}
	Let $\kotz$ be a $G_1$-space. Then the following statements are equivalent:
	\begin{enumerate}[label=\normalfont(\arabic*)]
		\item $\kotz$ is nuclear.
		\item For all $n \in \nn$, there exists $m>n$ such that
		\begin{equation}\label{nuclearity with i}
			\supi  \frac{ia_n(i)}{a_m(i)} < \infty.
		\end{equation}
		\item Given $\alpha \in \rn$, for all $n \in \nn$, there exists $m>n$ such that
		\begin{equation}\label{nuclearity with alpha}
			\supi  \frac{i^\alpha a_n(i)}{a_m(i)} < \infty.
		\end{equation}
	\end{enumerate}
\end{proposition}

\begin{proof}
	(1) $\Rightarrow$ (2) Let $\kotz$ be a nuclear $G_1$-space. Without loss of generality assume that $a_n(i)=1$ if $n \geq i$ and also $a_n(i) \leq a_{n+1}(i)^2$, for all $i,n$. Define $b_n(i):=1$ if $i<n$ and $b_n(i):=\prod_{j=n}^i a_j(i)$ if $i\geq n$. By \cite[Theorem 2]{Rob78} $\lambda_0(B)$ is regular and $\kotz=\lambda_0(B)$ both algebraically and topologically, that is, one has
	\begin{equation}\label{A equals B}
		\forall n \quad \exists m,\, C>0:\quad a_n(i) \leq Cb_m(i), \quad \forall i \in \nn, 
	\end{equation}
	and
	\begin{equation}\label{B equals A}
		\forall n \quad \exists m,\,D>0:\quad b_n(i) \leq Da_m(i),\quad \forall i \in \nn.
	\end{equation}
	Since $\lambda_0(B)$ is nuclear, fix $n \in \nn$, and pick $m>n$ such that $\sumi \frac{b_n(i)}{b_m(i)}<\infty$ by Grothendieck-Pietsch criterion. By regularity, $(\frac{b_n(i)}{b_m(i)})_i$ is decreasing. Then, by \cite[Theorem 3.3.1]{Kno56} for all $n \in \nn$ we can find $m>n$ such that $\limi \frac{ib_n(i)}{b_m(i)} = 0$, which implies
\begin{equation}\label{B satisfies}
	\supi  \frac{ib_n(i)}{b_m(i)} < \infty.
\end{equation}
	Now consider $a_n(i)$ for a fixed $n \in \nn$. By~(\ref{A equals B}) and~(\ref{B satisfies}), there exist $\tilde{m}$ and $M>0$ such that
\[
a_n(i) \leq C b_m(i) \leq CM\frac{b_{\tilde{m}}(i)}{i}, \quad \forall i \in \nn,
\]
Then by~(\ref{B equals A}) find $\tilde{n}>\tilde{m}$ such that
\[
\frac{b_{\tilde{m}}(i)}{i} \leq D \frac{a_{\tilde{n}}(i)}{i}, \quad \forall i \in \nn.
\]
Combining the arguments, one has
\[
	\supi  \frac{ia_n(i)}{a_{\tilde{n}}(i)} \leq CMD <\infty.
\]
	
	(2) $\Rightarrow$ (3) If $\alpha \leq 1$, it is trivial. Let $\alpha>1$. For any $n \in \nn$ we find $m_1 > n$ and $M_1>0$ such that $i a_n(i) \leq M_1 a_{m_1}(i)$, for all $i \in \nn$. Then one may find $m_2 > m_1$ and $M_2>0$ such that $i^2 a_n(i) \leq M_1M_2 a_{m_2}(i)$, for all $i \in \nn$. Continuing up to $k=\lfloor\alpha\rfloor+1$ one has $i^k a_n(i) \leq M_1 \dots M_k a_{m_{k}}(i)$, for all $i \in \nn$ so we obtain (3).
	
	(3) $\Rightarrow$ (1) Take $\alpha=2$. Given $n \in \nn$ apply (3) to find $m>n$ and $M>0$ with $i^2a_n(i) \leq Ma_m(i)$, for all $i \in \nn$. Then $\frac{a_n(i)}{a_m(i)}\leq \frac{M}{i^2}$, for all $i \in \nn$ and $(\frac{a_n	}{a_m}) \in \ell_1$. The space $\kotz$ is nuclear by the Grothendieck-Pietsch criterion. 
\end{proof}

\section{Spectrum of $\ces$ on $\kotz$}
For a locally convex Hausdorff space $X$ and $T \in \clo(X)$, the \textit{resolvent set} $\rho(X;T)$ of $T$ consists of all $\lambda \in \cn$ such that $(\lambda I-T)^{-1}$ exists in $\clo(X)$. The set $\sigma(T;X):=\cn\setminus\rho(T;X)$ is called the \textit{spectrum} of $T$ on $X$. The \textit{point spectrum} $\sigma_{pt}(T;X)$ of $T$ on $X$ consists of all $\lambda \in \cn$ such that $(\lambda I-T)$ is not injective. We denote $\Sigma:=\{\frac{1}{k}:k \in \nn\}$ and $\Sigma_0:=\Sigma \cup \{0\}$.

The inverse Cesàro operator $\ices\colon\cnn \to \cnn$ is given by the explicit formula
	\begin{equation}
		y=(y_i)_{i \in \nn} \mapsto \ices(y) = (iy_i-(i-1)y_{i-1})_{i \in \nn}, \quad y_0:=0.
	\end{equation}

\begin{proposition}\label{invertibility}
	Let $A$ be a Köthe matrix satisfying (G1-1). Then, the following statements are equivalent:
	 \begin{enumerate}[label=\normalfont(\arabic*)]
	 	\item $0 \notin \sigma(\ces;\kotz)$.
	 	\item For any $n \in \nn$, there exists $m>n$ such that
	\[
		\supi  \frac{ia_n(i)}{a_m(i)} < \infty.
	\]
\end{enumerate}
\end{proposition}

\begin{proof}
		(2) $\Rightarrow$ (1) For given $n \in \nn$ pick $m>n$ as in \eqref{nuclearity with i}. Then we have
	\begin{align*}
		a_n(i)|\ices(y)| \leq & a_n(i)|iy_i-(i-1)y_{i-1}|
					 \leq  ia_n(i)|y_i| + (i-1)a_n(i)|y_{i-1}| \\
					 \leq & ia_n(i)|y_i|+(i-1)a_n(i-1)|y_{i-1}| \\
					 =	  & \frac{ia_n(i)}{a_m(i)}a_m(i)|y_i|+\frac{(i-1)a_n(i-1)}{a_m(i-1)}a_m(i-1)|y_{i-1}|,
	\end{align*}
	for some $M>0$. Then
	\[
		\supi  a_n(i) |\ices y| \leq  M a_m(i)|y_i|+Ma_m(i-1)|y_{i-1}| \leq  2M\supi a_m(i)|y_i|.
	\]
	Therefore, $\ices$ is continuous.
	
	(1) $\Rightarrow$ (2) Let $\ices:\kotz \to \kotz$ be continuous. Then for all $n \in \nn$, there exists $m > n$ and $M>0$ with
	\[
		\supi  a_n(i)|iy_i-(i-1)y_{i-1}| \leq M \supi  a_m(i) |y_i|.
	\]
	Let $y=(e_j), j\geq 2$ to get $ja_n(j) \leq M a_m(j)$, for all $j \geq 2$, and hence
	\[
		\supi  \frac{ia_n(i)}{a_m(i)} \leq M < \infty. 
	\]
This completes the proof.
 
\end{proof}

\begin{lemma}\textup{\cite[Proposition 2.6]{ABRwc0}}\label{1/s in ptsp}
	The following statements are equivalent for the Cesàro operator $\ces$ defined on a Köthe echelon space of order zero $\kotz$ and $s \in \nn$.
	\begin{enumerate}[label=\normalfont(\arabic*)]
		\item $\frac{1}{s} \in \ptsp$.
		\item $\limi i^{s-1}a_n(i)=0$, for all $n\in \nn$.
	\end{enumerate}
\end{lemma}

\begin{proposition}\label{G11 supfinite}
	Let $A$ be a Köthe matrix satisfying the condition (G1-1). Then $\Sigma= \ptsp$ if for all $n \in \nn$, there exist $m>n$ and $M>0$ such that
	\[
		\supi  \frac{ia_n(i)}{a_m(i)} < \infty.
	\]
\end{proposition}

\begin{proof}
It is clear that $\ptsp \subseteq \sigma_{pt}(\ces;\cnn)=\Sigma$ \cite[Lemma 2.5(i)]{ABRwc0}. For the other inclusion, we prove that $i^sa_n(i)\igoes0$, for all $s,n \in \nn$ by induction on $s$.  Observe that $0 < ia_n(i) \leq M a_m(i) \igoes 0$ and hence $ia_n(i) \igoes 0$, for all $n \in \nn$. Let now $i^sa_n(i) \igoes 0$ for some $s$, and consider
\[
	i^{s+1}a_n(i)=i^s i a_n(i) \leq M i^sa_m(i) \igoes 0,
\]
by the induction hypothesis. That implies $i^{s+1}a_n(i) \igoes 0$, for all $s \in \nn$ and for all $n \in \nn$, which means $(i^s)_i \in \kotz$, for all $s \in \nn$, and hence by Lemma~\ref{1/s in ptsp}, $\frac{1}{s} \in \ptsp$ for all $s\in \nn$. So $\Sigma \subseteq \ptsp$. 
\end{proof}

\begin{theorem}\label{spectrum main theorem}
	For a $G_1$-space $\kotz$, the following statements are equivalent:
	\begin{enumerate}[label=\normalfont(\arabic*)]
		\item $0 \notin \spec$.
		\item $\frac{1}{2} \in \ptsp$.
		\item There exists $s \in \nn$, $s>1$ such that $\frac{1}{s} \in \ptsp$.
		\item $\Sigma=\ptsp$.
	\end{enumerate}
\end{theorem}

\begin{proof}
	(1) $\Rightarrow$ (2) By Proposition~\ref{invertibility}, for each $n \in \nn$ find $m>n$ and $M>0$ such that $ia_n(i)\leq Ma_m(i)\igoes0$. Hence $ia_n(i) \igoes 0$, for all $n \in \nn$, which implies $\frac{1}{2} \in \ptsp$ by Lemma~\ref{1/s in ptsp}.
	
	(2) $\Rightarrow$ (1) Fix $n \in \nn$. By (G1-2) find $m>n$ and $M>0$ such that $a_n(i) \leq M a_m(i)^2$, for all $i \in \nn$. Then, by Lemma~\ref{1/s in ptsp}
	\[
		\frac{ia_n(i)}{a_m(i)} \leq M i a_m(i) \igoes 0.
	\]
	Thus (1) follows by Propositions~\ref{nuclearity_criterion1} and~\ref{invertibility}.
	
	(2) $\Leftrightarrow$ (3) Clear.
	
	(3) $\Rightarrow$ (4) By Lemma~\ref{1/s in ptsp}, $\frac{1}{s} \in \ptsp$ if and only if $(i^{s-1}) \in \kotz$. Again by Lemma~\ref{1/s in ptsp}, $1 \in \ptsp$ since $(1,1,\dots) \in \kotz$. Assume there is $s \in \nn$ with $(i^s)_i \in \kotz$. Then $(i)_i \in \kotz$. We claim $\limi i^{2^k}a_n(i)=0$, for all $k,n \in \nn$. We prove it by induction on $k$ and all $n\in \nn$. For $k=0$, we have $\lim_i ia_n(i)=0$, for all $n \in \nn$. Suppose that $i^{2^r}a_n(i) \igoes 0$ for some $r \in \nn$, fix $n \in \nn$. By (G1-2) select $m>n$ and $M>0$ such that $a_n(i) \leq M a_m(i)^2$, for all $i \in \nn$. By the induction hypothesis $i^{2^r}a_m(i) \igoes 0$. Notice that $(i^{2^r}a_m(i))^2=i^{2^{r+1}}a_m(i)^2 \igoes 0$. Then
	\[
		0 < i^{2^{r+1}}a_n(i) \leq M i^{2^{r+1}} a_m(i)^2 \igoes 0.
	\]	
	Clearly $\lim_i i^{2^k}a_n(i)=0$, for all $k,n \in \nn$ implies $\limi i^t a_n(i)$, for all $t \in \nn$. Hence $\frac{1}{t} \in \ptsp$, for all $t \in \nn$.
	
	(4) $\Rightarrow$ (3) Trivial. 
	\end{proof}
	The following examples show that the assumption that $\kotz$ is a $G_1$-space in Theorem~\ref{spectrum main theorem} cannot be removed. We start with the following remark.
\begin{remark}\label{the phi for nong1}
\normalfont The function $\varphi(x)=x^\alpha e^{-x}$, $x \in (0,\infty)$, $\alpha>0$, is strictly decreasing on $(\alpha,\infty)$, has a local and global maximum at $x=\alpha$ and $\lim_{x \to 0^+}\varphi(x)=0$. In particular, if $0<\alpha<1$, then $\varphi$ is strictly decreasing on $[1,\infty)$.
\end{remark}

\begin{example}
\normalfont
\begin{enumerate}[wide, labelwidth=!, labelindent=0pt, label=(\textbf{\roman*})]
\setlength\itemsep{0.5em}	

\item Fix $0<\alpha<1$, take a strictly increasing sequence $(\alpha_n)_n \subset (0,\infty)$ tending to $\alpha$. Define the Köthe matrix $A=(a_n)_n$, $a_n(i)=i^{\alpha_n}e^{-i}$, $i,n \in \nn$. The Köthe echelon space $\kotz$ of order zero satisfies the condition (G1-1) and~(\ref{sequences vanish}) by Remark~\ref{the phi for nong1}. However, it is not a $G_1$-space. We first show that $0 \notin \spec$. Indeed, taking $n=1$, we get for each $m>1$,
\[
\supi  \frac{ia_1(i)}{a_m(i)}=\supi  i^{1-\alpha_m+\alpha_1}=\infty,
\]
since $1-\alpha_m+\alpha_1>0$. The condition follows from Proposition~\ref{invertibility}. On the other hand, for each $s,n\in \nn$, we have $0<i^{s-1}a_n(i)=i^{s+\alpha_n-1}e^{-i}\leq i^se^{-i}$, which tends to 0 as $i \to \infty$. Therefore $(i^{s-1})_i\in \kotz$ and $\frac{1}{s} \in \ptsp$ for each $s \in \nn$. This shows that condition (4) does not imply condition (1) in Theorem~\ref{spectrum main theorem} in general.

\item Fix $s \geq 1$, $s \in \nn$ and define the Köthe matrix $A=(a_n)_n$ by
\[
	a_n(i)=\frac{1}{i^{s-\frac{1}{2}+\frac{1}{n+1}}}, \quad i,n \in \nn.
\]
The Köthe echelon space $\kotz$ of order zero satisfies (G1-1) and~(\ref{sequences vanish}) but it is not a $G_1$-space. In this case $\ptsp=\{1,\frac{1}{2},\dots,\frac{1}{s}\}$, and condition (3) in Theorem~\ref{spectrum main theorem} does not imply condition (4). This is easy to show since $(i^{s-1})_i \in \kotz$, hence $(i^{m-1})_i \in \kotz$ for $m=1,2,\dots,s$ but $(i^s)_i\notin \kotz$, because $i^sa_1(i)=1$ for each $i \in \nn$. Thus $\frac{1}{s+1} \notin \ptsp$, which implies $\frac{1}{m} \notin \ptsp$ for each $m>s$, by Lemma~\ref{1/s in ptsp}. 
\end{enumerate}
\end{example}

The next lemma is well-known.

\begin{lemma}\label{spec_inclusion}
	Let $E$ be a Fréchet space, and let $T\colon E \to E$ be a continuous linear operator with the dual $T'\colon E' \to E'$. Then $\sigma_{pt}(T';E') \subset \sigma(T;E)$.
\end{lemma}

\begin{proof}
	Let $\lambda \in \sigma_{pt}(T';E')$. Then there exists a nonzero $u\in E'$ such that $T'u=\lambda u$. For $x \in E$,
	\[
	\langle (\lambda I-T)x,u \rangle =\langle x,(\lambda I-T')u \rangle = \langle x, \lambda u-T'u \rangle=0.
	\]
	If $(\lambda I-T)E$ is dense in $E$, then $u(z)=0$, for all $z \in E$. Contradiction. Therefore $(\lambda I-T)E$ is not dense in $E$, so $\lambda I-T$ is not surjective, and hence $\lambda \in \sigma(T;E)$. 
\end{proof}

\begin{proposition}\label{spectrum is sigma}
	Let $\kotz$ be a nuclear $G_1$-space. Then, $\Sigma = \spec$.
\end{proposition}

\begin{proof}
	The dual of the Cesàro operator $\ces'\colon \kotz' \to \kotz'$ is given by the formula
	\begin{equation}
		\ces'y = \left( \sum_{j=i}^\infty \frac{y_j}{j} \right)_{i \in \nn},
	\end{equation}
	for $y=(y_i)_{i \in \nn} \in \kotz' = \bigcup_{n=1}^\infty \ell_1(\frac{1}{a_n}) =  \{x \in \cnn| \, \exists n \in \nn: \sumi  \frac{|x_i|}{a_n(i)} < \infty\}$.
\begin{enumerate}[wide, labelwidth=!, labelindent=0pt, label=(\textbf{\roman*})]
\setlength\itemsep{0.5em}
\item $0 \notin \spec$: Follows by Proposition~\ref{invertibility}.

\item $\Sigma \subset \sigma_{pt}(\ces';\kotz')$: Let $\lambda \in \cn\setminus\{0\}$. Then $\ces'y=\lambda y$ for $y \in \kotz'$ if and only if $\lambda y_i = \sum_{j=i}^\infty \frac{y_j}{j}$, for all $i \in \nn$. The latter yields $\lambda(y_i-y_{i+1})=\frac{y_i}{i}$, and hence 
		  \begin{equation}\label{recurrence_relation}
		  y_{i+1} =  \left(1-\frac{1}{\lambda i}\right) y_i \quad \Rightarrow \quad
		  y_i     =  y_1 \prod_{j=1}^{i-1} \left(1-\frac{1}{\lambda j}\right), \, \, i \in \nn. 
		  \end{equation}
Now let $\lambda \in \Sigma$, that is, $\lambda=\frac{1}{s}$, for some $s \in \nn$. Let us define $y \in \kotz'$ by
		  \begin{equation}
		  	y_i = y_1 \prod_{j=1}^{i-1} \left(1-\frac{s}{j}\right),
		  \end{equation}
for $1 < i \leq s$ and $y_i=0$ for $i>s$ so that $y$ satisfies~(\ref{recurrence_relation}). Therefore $y$ belongs to the space of finitely supported sequences $c_{00}$, thus clearly to the space $\kotz'$, and satisfies $\ces'y = \lambda y$. Hence $ \lambda \in \sigma_{pt}(\ces';\kotz')$.
		  
\item $\Sigma \subseteq \spec$: follows by Lemma~\ref{spec_inclusion}.
		   
\item $\spec \subseteq \Sigma$: We show $(\ces-\lambda I)^{-1}$ is continuous on $\kotz$ for every $\lambda \in \cn \setminus \Sigma_0$. The $i$-th row of the matrix $(\ces-\lambda I)^{-1}$ is given by \cite{Rea85}:
		   \[
		   	\begin{cases}
		   	\displaystyle	\frac{-1}{i \lambda^2 \prod_{k=j}^i (1-\frac{1}{k\lambda})} & \text{if } 1 \leq j <i \vspace{0.25cm}\\
		   	\displaystyle	\frac{1}{\frac{1}{i}-\lambda} & \text{if } i=j \vspace{0.25cm}\\
		   		0 & \text{otherwise}.
		   	\end{cases}
		   \]
 For $D_\lambda=(d_{ij})_{i,j}$ and $E_\lambda=(e_{ij})_{i,j}$, one may formulate $(\ces-\lambda I)^{-1}=D_\lambda-\frac{1}{\lambda^2}E_\lambda$, where
		  \begin{multicols}{2}
 \noindent
		   \[
		   	d_{ij}= \begin{cases}
		   			\displaystyle \frac{1}{\frac{1}{i}-\lambda} & \text{if } i=j \vspace{0.25cm} \\
		   			0 & \text{otherwise}.
		   	\end{cases}	
		   \]
\noindent
		   \[
		   	e_{ij}= \begin{cases}
		   	\displaystyle	\frac{1}{i\prod_{k=j}^i \left(1-\frac{1}{k\lambda}\right)} & \text{if } 2 \leq j <i \vspace{0.25cm} \\
		   		0 & \text{otherwise}.
		   	\end{cases}
		   \]
\end{multicols}
By assumption, $\lambda \notin \Sigma_0$. Then $D_\lambda \in \clo(\kotz)$. $E_\lambda\colon \cnn \to \cnn$ acts continuously on $\kotz$ if and only if given $\lambda \in \cn\setminus \Sigma_0$ for every $n \in \nn$ there exists $m>n$ such that there exists a unique continuous extension $E^{nm}_\lambda\colon c_0(a_m) \to c_0(a_n)$. $E^{nm}_\lambda$ is continuous on $c_0(a_m)$ if and only if $\tilde{E}^{nm}_\lambda= \phi_n \circ E^{nm}_\lambda \circ \phi_m^{-1}$ is continuous on $c_0$, where $\phi_n(x)=(a_n(i)x_i)_i$, and $\tilde{E}^{nm}_\lambda=(\tilde{e}^{nm}_{ij})_{i,j}$. To prove that, we need to verify for each $n \in \nn$ \cite[Theorem 4.51-C]{Tay58}:
		   		\begin{enumerate}[label=\normalfont(\arabic*)]
		   			\item $\limi \tilde{e}^{nm}_{ij}=0$, for each $j \in \nn$.
		   			\item $\supi  \sumj  |\tilde{e}^{nm}_{ij}|<\infty$.
		   		\end{enumerate}
		   
Given $\alpha:=\operatorname{Re}(\frac{1}{\lambda}) \in \rn$ select $m>n$ by the nuclearity criterion~(\ref{nuclearity with alpha}) such that
		   \[
		   	\limi \frac{i^\alpha a_n(i)}{a_m(i)} = 0.
		   \]
For part (1), let $j \in \nn$ be fixed. Use the estimate in \cite[Lemma 7]{Rea85} to see
		   \[
		   	|\tilde{e}^{nm}_{ij}| \leq C \frac{a_n(i)}{a_m(j)} \frac{1}{i^{1-\alpha}j^\alpha} = \frac{C}{a_m(j)j^\alpha}i^{\alpha-1}a_n(i) \leq \frac{C'}{a_m(j)j^\alpha}\frac{a_m(i)}{i} < \varepsilon, \quad \forall i>i_0, 
		   \] 
 for any given $\varepsilon>0$ and for some $i_0 \in \nn$ where $C,C'>0$. For part (2), consider
		   \[
		   	\sumj  |\tilde{e}^{nm}_{ij}|=\sum_{j=1}^{i-1} \frac{a_n(i)}{a_m(j)}|e_{ij}| \leq C \sum_{j=1}^{i-1} \frac{a_n(i)}{a_m(j)} \frac{1}{i^{1-\alpha}j^\alpha} = C \frac{1}{i}\sum_{j=1}^{i-1} \frac{a_n(i) i^\alpha}{a_m(j)j^\alpha},
		   \]
for some $C>0$. Now let $\alpha<1$. Observing $0<\frac{a_n(i)}{a_m(i)} \leq 1$, for all $i \in \nn$ and for all $m>n$, 
		   \[
		   	\sumj  |\tilde{e}^{nm}_{ij}| \leq \frac{C}{i^{1-\alpha}} \sum_{j=1}^{i-1} \frac{1}{j^\alpha} \leq C \max \left(1, \frac{1}{1-\alpha}\right) < \infty,
		   \]
where the last inequality is due to the proof of \cite[Corollary 3.6]{ABRwc0}. If $\alpha\geq 1$, then
		   \[
		   	\sumj  |\tilde{e}^{nm}_{ij}| \leq C \frac{i^\alpha}{i}a_n(i) \sum_{j=1}^{i-1}\frac{1}{a_m(j)j^\alpha} \leq C \frac{i^\alpha}{i}a_n(i)(i-1)\frac{1}{a_m(i)} \leq C \frac{i^\alpha a_n(i)}{a_m(i)} <\infty,
		   \]
Therefore, we deduce that $E_\lambda$ is continuous on $\kotz$ which implies that $(\ces-\lambda I)^{-1}$ is continuous for every $\lambda \in \cn \setminus \Sigma_0$. 
\end{enumerate}
\end{proof}

\begin{corollary}
	Let $\ces\colon \kotz \to \kotz$, where $\kotz$ is a nuclear $G_1$-space. Then $\ces$ is neither compact nor weakly compact.
\end{corollary}

\begin{proof}
	Since $\kotz$ is a $G_1$-space with $a_n(i) \igoes 0$, for all $n \in \nn$, $\kotz$ is a Schwartz space. In particular, it is a Montel space. So there is no distinction between compactness and weak compactness in $\kotz$. Now let $\ces$ be compact. Then, by \cite[Theorem 9.10.2(4)]{Edw65}, $\spec$ is necessarily a compact set. However, that contradicts Proposition~\ref{spectrum is sigma}. 
\end{proof}

For the Köthe matrix $A=(a_n)_n$, define
\begin{equation}\label{sn definition}
	S_n(A):=\left\{s \in \rn: \sumi  \frac{1}{i^s a_n(i)} < \infty \right\}, \quad n \in \nn.
\end{equation}
It follows from~(\ref{sn definition}) and (K1) that $S_n(A) \subseteq S_{n+1}(A)$, for all $n \in \nn$. If $S_{n_0}(A) \neq \varnothing$ for some $n_0 \in \nn$, then $S_n(A) \neq \varnothing$ for all $n \geq n_0$. However, as shown in Example~\ref{Sn counterexample}, $S_{n_0}(A) \neq \varnothing$ does not always imply that $S_n(A) \neq \varnothing$ for all $n<n_0$.

\begin{example}\label{Sn counterexample}
\normalfont	Let $a_n(i)=\exp(-e^{\alpha_i/n})$, where $\alpha_i=2\log(\log(i+2))$. Assume $S_1(A) \neq \varnothing$, then there exists $s \in \rn$ and $M>0$ such that $\frac{\exp(e^{\alpha_i})}{i^s} \leq M$, for all $i \in \nn$. Then,
	\begin{align*}
		e^{2\log(\log(i+2))}  & \leq \log(M)+ s \log(i+2) \\
		(\log(i+2))^2 & \leq \log(M)+ s \log(i+2) \\
		\underbrace{\log(i+2)}_{\igoes\infty} & \leq \underbrace{\frac{\log(M)}{\log(i+2)}}_{\igoes0} +\, s,
	\end{align*}
	for all $i \in \nn$. This is obviously a contradiction, and hence $S_1(A)=\varnothing$. However, for $s=2+\varepsilon$,
	\[
		\frac{1}{i^s a_2(i)}=\frac{\exp(e^{\alpha_i/2})}{i^{2+\varepsilon}} = \frac{\exp(\log(i+2))}{i^{2+\varepsilon}} = \frac{i+2}{i^{2+\varepsilon}} \simeq \frac{1}{i^{1+\varepsilon}} \in \ell_1,
	\]
	for any given $\varepsilon>0$. Then,
	\[
		\sumi  \frac{1}{i^{2+\varepsilon}a_2(i)} < \infty,
	\]
	which implies $S_2(A)\neq \varnothing$. 
\end{example}

For some $n_0 \in \nn$, let $S_{n_0}\neq \varnothing$. We define $s_0(n):=\inf S_n(A)$. We have $s_0(n+1) \leq s_0(n)$, for all $n\geq n_0$. Moreover, since $a_n(i) \downarrow 0$, there is $i_0 \in \nn$ such that
\[
	\sum_{i=i_0}^\infty \frac{1}{i^t a_n(i)} \geq \sum_{i=i_0}^\infty \frac{1}{i^t},
\]
for all $t \in S_n(A)$. This also shows that, if $A$ is a (G1-1) Köthe matrix with $S_n(A) \neq \varnothing$, then $s_0(n) \geq 1$.

\begin{lemma}\label{Sn reduces to limit}
	For a fixed $n \in \nn$ and a Köthe matrix $A$, the following statements are equivalent:
	\begin{enumerate}[label=\normalfont(\arabic*)]
	\setlength\itemsep{0.5em}
		\item $S_n(A) \neq \varnothing$.
		\item There exists $t>1$ such that $[t,\infty) \subset S_n(A)$.
		\item There exists $t>1$ such that for all $\displaystyle s \in [t,\infty), \,\limi \frac{1}{i^s a_n(i)}=0$.
		\item There exists $t>1$ such that for all $\displaystyle s \in [t,\infty),\, \limi \frac{1}{i^s a_n(i)}<\infty$.
	\end{enumerate}
\end{lemma} 

\begin{proof}
	(1) $\Rightarrow$ (2) $\Rightarrow$ (3) $\Rightarrow$ (4) is trivial.
	
	(4) $\Rightarrow$ (1) Fix a real $t>1$ such that for all $s \geq t$ we have $\supi \frac{1}{i^sa_n(i)}<\infty$. Then, if $s>t+2$ it follows that
	\[
		\sumi \frac{1}{i^sa_n(i)} \leq \sumi \frac{1}{i^{t+2}a_n(i)} = \sumi \frac{1}{i^2 \, i^t a_n(i)} \leq M \sumi  \frac{1}{i^2} < \infty,
	\]
	since there exists $M>0$ such that $\frac{1}{i^ta_n(i)} \leq M$, for all $i \in \nn$. This implies $(t+2,\infty)\subset S_n(A)$. 
\end{proof}

\begin{lemma}\label{equivalences on Sn}
	Let $A$ be a Köthe matrix with $a_n(i):=\exp(-f(\alpha_i/n))$, where $\alpha=(\alpha_i)_i \uparrow \infty$, and let $f$ be a strictly increasing real function. Then, for a fixed $n \in \nn$, the following statements are equivalent:
	\begin{enumerate}[label=\normalfont(\arabic*)]
	\setlength\itemsep{0.5em}
		\item $S_n(A) \neq \varnothing.$	
		\item There exist $t>1$ and $M>0$ such that $\displaystyle \exp(f(\alpha_i/n)) \leq M i^t.$
		\item $\displaystyle \supi  \frac{f(\alpha_i/n)}{\log(i)} < \infty.$
		\item $\displaystyle \supi  \frac{\alpha_i}{f^{-1}(\log(i))} < \infty.$
	\end{enumerate}
\end{lemma}

\begin{proof}
	(1) $\Leftrightarrow$ (2) Follows by Lemma~\ref{Sn reduces to limit}.

	(2) $\Rightarrow$ (3) Since $f(\alpha_i/n) \leq \log(Mi^t)$. Then for $\tilde{M}:=M^{1/t}$ we have,
	\[
		\frac{f(\alpha_i/n)}{\log(i)} \leq t \frac{\log(\tilde{M}i)}{\log(i)}= t \left(\frac{\log(\tilde{M})}{\log(i)}+1 \right) \igoes t < \infty.
	\]
	(3) $\Rightarrow$ (2) Since there exists $M>0$ such that $f(\alpha_i/n) \leq M \log(i) = \log(i)^M$, for all $i \in \nn$, we have $\exp(f(\alpha_i/n)) \leq C i^M$, for $C>0$ which implies (2).
	
	(3) $\Leftrightarrow$ (4) $\sup_i \frac{\alpha_i}{f^{-1}(\log(i))} < \infty$ if and only if there exists $C \in \nn$, $C>0$ such that $\alpha_i \leq \frac{C}{n} n f^{-1}(\log(i))$, for all $i \in \nn$ if and only if $f(\alpha_i/n) \leq \frac{C}{n}\log(i)$, for all $i \in \nn$ if and only if $\sup_i \frac{f(\alpha_i/n)}{\log(i)}<\infty$.  
\end{proof}

\begin{remark}\label{more G1}
\normalfont
Let $f$ defined in Lemma~\ref{equivalences on Sn} be an odd and logarithmically convex in addition. Then the Köthe echelon space $\kotz$ of order 0 for which $A$ is as in Lemma~\ref{equivalences on Sn} is a $G_1$-space if there exists $k \in \nn$ such that $2f(x) \leq f(kx)$ for all $x \geq 1$. In this case, $\kotz=:L_f(\alpha_i,0)$ is called a \textit{Dragilev space of finite type}.
\end{remark}

\begin{lemma}\label{lemma on Sn}
	Let $\kotz$ be a $G_1$-space. If $S_{n_0}(A) \neq \varnothing$, for some $n_0 \in \nn$, then $\kotz$ is not nuclear.
\end{lemma}
\begin{proof}
If $S_{n_0}(A) \neq \varnothing$, there exists $t \in \rn$ such that 
	\begin{equation}\label{sn assumption}
		\sumi  \frac{1}{i^t a_{n_0}(i)}<\infty.
	\end{equation} 
	Suppose that $\kotz$ is nuclear. Then, by Proposition~\ref{nuclearity_criterion1} we find $m>n_0$ and $M>0$ such that
	$\frac{i^ta_{n_0}(i)}{a_m(i)}<M$. Hence
	\[
		\sumi  \frac{1}{i^ta_{n_0}(i)} \geq \frac{1}{M} \sumi  \frac{1}{a_m(i)} \xrightarrow{} \infty,
	\]
	which contradicts~(\ref{sn assumption}). So $\kotz$ is not nuclear. 
\end{proof}

\begin{lemma}\label{G1 arbitrary}
	Let $\kotz$ be a $G_1$-space. Then, for each $n \in \nn$ and for all $k \in \nn$ there exists $m>n$ and $M>0$ such that
	\[
		\frac{1}{a_m(i)^k} \leq M \frac{1}{a_n(i)}, \quad \forall i\in \nn.
	\]
\end{lemma}

\begin{proof}
	Let us prove the assertion for $k=2^s$ by induction over $s$. First observe that (G1-1) implies that there exists $i_0 \in \nn$ such that for any $n \in \nn$, $\frac{1}{a_n(i)}>1$, for all $i \geq i_0$. For $s=1$, for each $n \in \nn$ we find $m>n$ and $M>0$ with
	\[
		\frac{1}{a_m(i)^2} \leq M \frac{1}{a_n(i)}, \quad \forall i \in \nn,
	\]
	by condition (G1-2). Assume that $\frac{1}{a_m(i)^{2^s}} \leq M \frac{1}{a_n(i)}$, for all $i \in \nn$. For some $\tilde{m}>m$, it follows by (G1-2) and the induction hypothesis that
	\[
		\frac{1}{a_{\tilde{m}}(i)^{2^{s+1}}} = \left(\frac{1}{a_{\tilde{m}}(i)^{2^s}}\right)^2 \leq \tilde{M} \frac{1}{a_m(i)^{2^s}} \leq \overline{M} \frac{1}{a_n(i)}, \quad \forall i \geq i_0, 
	\]
	where $\tilde{M}, \, \overline{M}>0$. 
\end{proof}

\begin{lemma}\label{infimum of SnA}
	Let $A$ be a $G_1$-Köthe matrix. If there exists $n_0 \in \nn$ with $S_{n_0}(A) \neq \varnothing$, then $s_0(A):=\inf_n\{\inf S_n(A)\}=1$.
\end{lemma}

\begin{proof}
	Let $n_0 \in \nn$ with $S_{n_0}(A) \neq \varnothing$. Due to the remarks prior to Lemma~\ref{Sn reduces to limit} we only need to show that $s_0(n) \leq 1$. For some $t \in \rn$, $\sum_{i \in \nn} \frac{1}{i^t a_{n_0}(i)}<\infty$. Then there exists $i_0 \in \nn$ such that $\frac{1}{a_{n_0}(i)}<i^t$, for all $i \geq i_0$. Fix $\varepsilon>0$ and take $n_1 \in \nn$ with $\frac{t}{n_1}<\varepsilon$. By Lemma~\ref{G1 arbitrary}, we find $m>n_0$ and $M>0$ such that $\frac{1}{a_{m}(i)^{n_1}} \leq M \frac{1}{a_{n_0}(i)}$, for all $i \geq i_0$. Hence we have
	\[
		\frac{1}{a_m(i)} \leq M^{1/n_1} \frac{1}{a_{n_0}(i)^{1/n_1}} \leq \tilde{M} i^{t/n_1} < \tilde{M} i^\varepsilon, \quad \forall i \geq i_0.
	\]
	Now choose $0<\tilde{\varepsilon}<\varepsilon$ with $\frac{1}{a_m(i)} \leq \tilde{M}i^{\tilde{\varepsilon}}$. Therefore
	\[
		\sum_{i=i_0}^\infty \frac{1}{i^{1+\varepsilon}a_m(i)} \leq \tilde{M} \sum_{i=i_0}^\infty \frac{1}{i^{1+\varepsilon-\tilde{\varepsilon}}}< \infty,
	\]
	which means $s_0(m) \leq 1+\varepsilon$ implying that $s_0(A)=1$. 
\end{proof}

For each $r \geq 1$, define the open disk $D(r):=\{\lambda \in \cn:|\lambda-\frac{1}{2r}|<\frac{1}{2r}\}$. It is routine to establish that if $|\lambda-\frac{1}{2s_0(n)}|<\frac{1}{2s_0(n)}$, then $\operatorname{Re}(\frac{1}{\lambda}) > s_0(n)$.

\begin{proposition}
	Let $\kotz$ be a $G_1$-space with $S_{n_0}(A) \neq \varnothing,$ for some $n_0 \in \nn$. Then
	\begin{equation}\label{nonnuclear spectrum}
		D(1) \cup \{1\} = \bigcup_{n \in \nn} D(s_0(n)) \cup \Sigma \subseteq \spec.
	\end{equation}
\end{proposition}

\begin{proof}
	Lemma~\ref{infimum of SnA} yields $s_0(A)=1$. We get $\Sigma \subseteq \spec$ directly by Proposition~\ref{spectrum is sigma} part (iii) since the proof is independent of nuclearity. Fix $n \in \nn$ and let $\lambda \in \cn \setminus \Sigma$ satisfy $|\lambda-\frac{1}{2s_0(n)}|<\frac{1}{2s_0(n)}$ so that $\lambda \in D(s_0(n))\setminus \Sigma$. For any $y_1 \in \cn\setminus \{0\}$ define $y \in \cnn\setminus \{0\}$ by $y_{i+1}:=y_1 \prod_{j=1}^i (1-\frac{1}{j\lambda})$ for $i \in \nn$. Then $\ces'_ny=\lambda y$, where $\ces'_n\colon \ell_1(1/a_n)\to \ell_1(1/a_n)$ is the dual of the Cesàro operator $\ces_n\colon c_0(a_n) \to c_0(a_n)$. Moreover,
	\[
	\sumi  \frac{|y_i|}{a_n(i)} = \frac{|y_1|}{a_n(1)}+|y_1|\sum_{i=2}^\infty \prod_{j=1}^i \left|1-\frac{1}{j\lambda}\right|\frac{1}{a_n(i)} \leq \frac{|y_1|}{a_n(1)}+C|y_1| \sum_{i=2}^\infty\frac{1}{i^\alpha a_n(i)},
	\]
	where $\alpha=\operatorname{Re}(\frac{1}{\lambda}),\, C>0$, and the inequality is due to the estimate in \cite[Lemma 7]{Rea85}. But since $\alpha>s_0(n) \geq 1$, we have $\alpha \in S_n(A)$ and hence $y \in \kotz'$. For $z \in \kotz \subseteq c_0(a_n)$ the vector $(\ces-\lambda I)z$ with $\ces \in \clo(\kotz)$ belongs to $c_0(a_n)$. Since $y \in c_0(a_n)'\subseteq \kotz'$ it follows that $\langle (\ces-\lambda I)z,y \rangle=\langle (\ces_n-\lambda I)z,y \rangle=\langle z, (\ces'_n-\lambda I)y \rangle=0$. Thus $\langle u,y \rangle=0$ for every $u \in \overline{\operatorname{Im}(\ces-\lambda I)} \subseteq \kotz$ with $y \in \kotz'\setminus \{0\}$. Then $\ces-\lambda I \in \clo(\kotz)$ is not surjective and hence $\lambda \in \spec$. This implies $D(s_0(n)) \setminus \Sigma \subseteq \spec$ and so $\Sigma \cup D(s_0(n)) \subseteq \spec$. Since $n \in \nn$ was arbitary and $\lim_n \frac{1}{2s_0(n)}=\frac{1}{2}$, by Lemma~\ref{infimum of SnA} we establish~(\ref{nonnuclear spectrum}). 
\end{proof}

\section{Iterates of $\ces$ and mean ergodicity}
Let $X$ be a Fréchet space topologized by the increasing system of seminorms $(p_n(\cdot))_{n \in \nn}$. Then, the strong operator topology $\tau_s$ in $\clo(X)$ is determined by the seminorms
\[
p_{n,x}(S):= p_n(Sx), \quad \forall x \in X, \, \forall n \in \nn,
\]
where $S \in \clo(X)$. In this case we write $\clo_s(X)$. Let $\mathcal{B}(X)$ be the family of bounded subsets of $X$. Then, the uniform topology $\tau_b$ in $\clo(X)$ is defined by the family of seminorms
\[
p_{n,B}(S):=\sup_{x \in B} p_n(Sx), \forall n \in \nn, \, \forall B \in \mathcal{B}(X),
\]
where $S \in \clo(X)$. In this case we write $\clo_b(X)$. $\tau_b$ is called the topology of uniform convergence on bounded subsets of $X$.

An operator $T \in \clo(X)$, where $X$ is a Fréchet space, is \textit{power bounded} if $(T^k)_{k=1}^\infty$ is an equicontinuous subset of $\clo(X)$. Given $T \in \clo(X)$, the averages
\[
	T_{[k]}:=\frac{1}{k} \sum_{j=1}^k T^j, \quad k \in \nn,
\]
are called the \textit{Cesàro means} of $T$. The operator $T$ is said to be \textit{mean ergodic} (resp., \textit{uniformly mean ergodic}) if $(T_{[k]})_k$ is a convergent sequence in $\clo_s(X)$ (resp., in $\clo_b(X)$). A Fréchet space $X$ is called mean ergodic if every power bounded operator on $X$ is mean ergodic.
\begin{proposition}\label{mean ergodic}
	Let $\kotz$ be a $G_1$-space. Then $\ces\colon \kotz \to \kotz$ is power bounded and uniformly mean ergodic. In particular, 
\begin{equation}
	\kotz=\ker(I-\ces) \oplus \overline{(I-\ces)(\kotz)}. 
\end{equation}
	Moreover, $\ker(I-\ces)=\operatorname{span}\{\boldsymbol{1}\}$ and 
\begin{equation}	
	\overline{(I-\ces)(\kotz)}=\{x \in \kotz:x_1=0\}=\overline{\operatorname{span}(e_i)_{i \geq 2}}.
\end{equation}
\end{proposition}

\begin{proof}
	Let $\ces^k$ denote the $k$-th iterate of $\ces$. Pick an arbitary $x \in \kotz$. For $k=1$, one has $p_n(\ces x) \leq p_n(x)$ for any $n \in \nn$. Suppose that $p_n(\ces^r x) \leq p_n(x)$ for some $r \in \nn$. Then,
	\[
		p_n(\ces^{r+1}x) = p_n(\ces(\ces^rx)) \leq p_n(\ces^rx) \leq p_n(x).
	\]
	This implies $(\ces^k)_{k=1}^\infty$ is equicontinuous, hence $\ces$ is power bounded on $\kotz$. But since $\kotz$ is a $G_1$-space with $a_n(i) \igoes 0$, for all $n \in \nn$, $\kotz$ is a Schwartz space. In particular, it is a Montel space and hence a mean ergodic space \cite[Proposition 2.13]{ABR09}. Hence $\ces$ is uniformly mean ergodic. The facts that $\ker(I-\ces)=\operatorname{span}\{\boldsymbol{1}\}$ and $\overline{(I-\ces)(\kotz)}=\{x \in \kotz:x_1=0\}=\overline{\operatorname{span}(e_i)_{i \geq 2}}$ follow by applying the same arguments used in \cite[Proposition 4.1]{ABR13}. 
\end{proof}

\begin{proposition}
	Let $\kotz$ be a nuclear $G_1$-space. Then $(I-\ces)^m(\kotz)$ is a closed subspace of $\kotz$ for each $m \in \nn$.
\end{proposition}

\begin{proof}
	Using Proposition~\ref{mean ergodic}, we proceed as in the proof of the analogous result when $\ces$ acts on the power series space of finite type, $\Lambda_0(\alpha)$ \cite[Proposition 3.4]{ABR18}. Consider first $m=1$. By Proposition~\ref{mean ergodic} we have $\overline{(I-\ces)(\kotz)}=\{x \in \kotz:x_1=0\}$. Set $X_1(A):=\{x \in \kotz:x_1=0\} \subseteq \kotz$. Clearly $(I-\ces)(\kotz)\subseteq X_1(A)$. We claim that $(I-\ces)(X_1(A))=(I-\ces)(\kotz)$. One inclusion is obvious. To establish the other, observe that
	\begin{equation}\label{(I-C)x}
		(I-\ces)x=\left(0,x_2-\frac{x_1+x_2}{2},x_3-\frac{x_1+x_2+x_3}{3}, \dots\right), \, x=(x_i)_i \in \kotz,
	\end{equation}
	and in particular
	\begin{equation}\label{(I-C)y}
		(I-\ces)y=\left(0, \frac{y_2}{2}, y_3-\frac{y_2+y_3}{3}, y_4-\frac{y_2+y_3+y_4}{4}, \dots \right), \, y=(y_i)_i \in X_1(A).
	\end{equation}
	Fix $x \in \kotz$. It follows from~(\ref{(I-C)x}) that
	\begin{equation}\label{xj}
		x_j-\frac{1}{j}\sum_{k=1}^jx_k=\frac{1}{j}\left((j-1)x_j-\sum_{k=1}^{j-1}x_k\right), \quad j \geq 2,
	\end{equation}
	is the $j$-th coordinate of the vector $(I-\ces)x$. Set $y_i:=x_i-x_1$ for all $i \in \nn$ and observe that $y \in X_1(A)$ because $(0,1,1,1,\dots) \in \kotz$. Now, by~(\ref{(I-C)y}), the $j$-th coordinate of $(I-\ces)y$ is given by~(\ref{xj}) for $j \geq 2$. So $(I-\ces)x=(I-\ces)y \in (I-\ces)(X_1(A))$.
	
	To show that $(I-\ces)(\kotz)$ is closed in $\kotz$ is suffices to prove that the continuous linear operator $(I-\ces)|_{X_1(A)}\colon X_1(A) \to X_1(A)$ is bijective. It is clearly injective by~(\ref{(I-C)y}), so we prove it is surjective. Observe $X_1(A)=\bigcap_{n \in \nn} X^{(n)}$ with $X^{(n)}:=\{x \in c_0(a_n):x_1=0\}$ is a closed subspace of $c_0(a_n)$ for all $n \in \nn$. Set $\tilde{a}_n(i):=a_n(i+1)$, and $\tilde{A}=(\tilde{a}_n)_n$, for all $i, \, n \in \nn$. Then $X_1(A)$ is topologically isomorphic to $\lambda_0(\tilde{A}):=\bigcap_{n \in \nn} c_0(\tilde{a}_n)$ via the left shift operator $S\colon X_1(A) \to \lambda_0(\tilde{A})$ given by $S(x)=(x_2,x_3,\dots)$ for $x=(x_1,x_2,\dots) \in X_1(A)$. We claim that the operator $T\colon =S \circ (I-\ces)|_{X_1(A)} \circ S^{-1} \in \clo(\lambda_0(\tilde{A}))$ is bijective with $T^{-1} \in \clo(\lambda_0(\tilde{A}))$.
	
	To verify the claim, first we observe that the operator $T\colon \cnn \to \cnn$ is bijective, and its inverse $R:=T^{-1}\colon \cnn \to \cnn$ is determined by direct calculation as a lower triangular matrix $R=r_{ij}$ with entries
	\[
		r_{ij}=
		\begin{cases}
			\displaystyle \frac{1}{j} & \text{ if } 1 \leq j < i\vspace{0.25cm} \\
			\displaystyle \frac{i+1}{i} & \text{ if } i=j\vspace{0.25cm} \\
			0 & \text{ if } j>i.
		\end{cases}
	\]
	To show that $R$ is also the inverse of $T$ on $\lambda_0(\tilde{A})$, we need to verify $R \in \clo(\lambda_0(\tilde{A}))$. This is equivalent to show that for each $n \in \nn$ there exists $m>n$ such that $\phi_n \circ R \circ \phi^{-1}_m \in \clo(c_0)$, where the operator $\phi_k\colon c_0(\tilde{a}_k) \to c_0$ is given by $\phi_k(x)=(\tilde{a}_k(i)x_i)_i$ for $x \in c_0(\tilde{a}_k)$ is a surjective isometry. Whenever $n \in \nn$ and $m>n$, the lower triangular matrix corresponding to $\phi_n \circ R \circ \phi^{-1}_m$ is given by $D_{mn}:=\left(\frac{a_n(i+1)}{a_m(j+1)}r_{ij}\right)_{i,j}$. For each fixed $j \in \nn$, $D_{mn}$ satisfies
	\[
		\limi \frac{a_n(i+1)}{a_m(j+1)}r_{ij}=\frac{1}{ja_m(j+1)}\limi a_n(i+1)=0.
	\]
	Moreover, by (G1-1) and the nuclearity criterion~(\ref{nuclearity with i}),
	\begin{align*}
		\sumj  \frac{a_n(i+1)}{a_m(j+1)}r_{ij} & =\sum_{j=1}^{i-1} \frac{a_n(i+1)}{ja_m(j+1)}+\frac{i+1}{i}\frac{a_n(i+1)}{a_m(i+1)} \\
		& \leq \frac{a_n(i+1)}{a_m(i+1)} \sum_{j=1}^{i-1}\frac{1}{j}+\frac{i+1}{i}\frac{a_n(i+1)}{a_m(i+1)} \\
		& = \frac{a_n(i+1)}{a_m(i+1)}\left(\sum_{j=1}^{i-1}\frac{1}{j}+\frac{i+1}{i}\right) \\
		& \leq (i+1) \frac{a_n(i+1)}{a_m(i+1)} < \infty,
	\end{align*}
	Thus, by \cite[Theorem 4.51-C]{Tay58}, $\phi_n \circ R \circ \phi^{-1}_m \in \clo(c_0)$, as desired. Hence $(I-\ces)(\kotz)$ is closed in $\kotz$. Since $(I-\ces)(\kotz)$ is closed, it follows from Proposition~\ref{mean ergodic} that $\kotz=\ker(I-\ces)\oplus(I-\ces)(\kotz)$. Hence, by \cite[Remark 3.6]{ABR13}, $(I-\ces)^m(\kotz)$ is closed for all $m \in \nn$.   
\end{proof}

A Fréchet space operator $T \in \clo(X)$, where $X$ is separable, is called \textit{hypercyclic} if there exists $x \in X$ such that the orbit $\{T^kx:k \in \nn_0\}$ is dense in $X$. If, for some $z \in X$, the projective orbit $\{\lambda T^kz:\lambda \in \cn, k \in \nn_0\}$ is dense in $X$, then $T$ is called \textit{supercyclic}. Clearly, if $T$ is hypercyclic then $T$ is supercyclic.

\begin{proposition}
	Let $A$ be a Köthe matrix. Then $\ces \in \clo(\kotz)$ is not supercyclic, and hence not hypercyclic either.
\end{proposition}

\begin{proof}
	Follows from \cite[Proposition 4.13]{ABRwc0}. 
\end{proof}

\section*{Acknowledgements}
The author wishes to thank Prof. José Bonet for crucial suggestions and discussions. He is also thankful to the anonymous referees as well as Prof. David Jornet for their careful reviewing.
\bibliography{cesaro}

\begin{thebibliography}{10}
\providecommand{\url}[1]{{#1}}
\providecommand{\urlprefix}{URL }
\expandafter\ifx\csname urlstyle\endcsname\relax
  \providecommand{\doi}[1]{DOI~\discretionary{}{}{}#1}\else
  \providecommand{\doi}{DOI~\discretionary{}{}{}\begingroup
  \urlstyle{rm}\Url}\fi

\bibitem{ABR09}
Albanese, A., Bonet, J., Ricker, W.: Mean ergodic operators in {F}r{\'e}chet
  spaces.
\newblock Ann. Acad. Sci. Fenn. Math. \textbf{34}, 401--436 (2009)

\bibitem{ABR13}
Albanese, A., Bonet, J., Ricker, W.: Convergence of arithmetic means of
  operators in {F}r{\'e}chet spaces.
\newblock J. Math. Anal. Appl. \textbf{401}, 160--173 (2013)

\bibitem{ABRwc0}
Albanese, A., Bonet, J., Ricker, W.: Mean ergodicity and spectrum of the
  {C}es{\`a}ro operator on weighted {$c_0$} spaces.
\newblock Positivity \textbf{20}, 761--803 (2016)

\bibitem{ABR17}
Albanese, A., Bonet, J., Ricker, W.: The {C}es{\`a}ro operator in the
  {F}r{\'e}chet spaces {$\ell^{p+}$} and {$L^{p-}$}.
\newblock Glasg. Math. J. \textbf{59}(2), 273--287 (2017)

\bibitem{ABR18}
Albanese, A., Bonet, J., Ricker, W.: The {C}es{\`a}ro operator on power series
  spaces.
\newblock Studia Math. \textbf{240}(1), 47--68 (2018)

\bibitem{ABR18-2}
Albanese, A., Bonet, J., Ricker, W.: Operators on the {F}r{\'e}chet sequence
  spaces {$ces(p+)$}, {$1 \leq p < \infty$}.
\newblock Rev. R. Acad. Cienc. Exactas F{\'\i}s. Nat. Ser. A Math. RACSAM pp.
  1--24 (2018)

\bibitem{Bra88}
Braun, R.W.: Linear topological structure of closed ideals in weighted algebras
  of entire functions.
\newblock Arch. Math. (Basel) \textbf{50}, 251--258 (1988)

\bibitem{Dra70}
Dragilev, M.M.: On regular bases in nuclear spaces.
\newblock Amer. Math. Soc. Trans. \textbf{93}, 61--82 (1970)

\bibitem{Edw65}
Edwards, R.E.: Functional Analysis. Theory and Applications.
\newblock Holt, Rinehart and Winston, New York Chicago San Francisco (1965)

\bibitem{Kno56}
Knopp, K.: Infinite Sequences and Series.
\newblock Dover, New York (1956)

\bibitem{Koc85}
Kocatepe, M.: On {D}ragilev spaces and the functor {E}xt.
\newblock Arch. Math. (Basel) \textbf{44}, 438--445 (1985)

\bibitem{Koc88}
Kocatepe, M.: Classification of {D}ragilev spaces of types-1 and 0.
\newblock Math. Balkanica \textbf{2}(2-3), 266--275 (1988)

\bibitem{Koc89}
Kocatepe, M., Nurlu, Z.: Some special {K}{\"o}the spaces.
\newblock In: T.~Terzio{\~g}lu (ed.) Advances in the Theory of Fr{\'e}chet
  Spaces, \emph{Series C: Mathematical and Physical Sciences}, vol. 287, pp.
  269--296. NATO Advanced Research Workshop on Advances in the Theory of
  Fr{\'e}chet Spaces, Kluwer, Dordrecht (1989)

\bibitem{Vog97}
Meise, R., Vogt, D.: Introduction to Functional Analysis.
\newblock No.~2 in Oxford Graduate Texts in Mathematics. Clarendon Press,
  Oxford (1997)

\bibitem{Ram79}
Ramanujan, M.S., Terzio{\~g}lu, T.: Subspaces of smooth sequence spaces.
\newblock Studia Math. \textbf{65}, 299--312 (1979)

\bibitem{Rea85}
Reade, J.B.: On the spectrum of the {C}es{\`a}ro operator.
\newblock Bull. Lond. Math. Soc. \textbf{17}, 263--267 (1985)

\bibitem{Rob78}
Robinson, W.: Some equivalent classes of {K}{\"o}the spaces.
\newblock Comment. Math. Prace Mat. \textbf{20}(2), 449--451 (1978)

\bibitem{Tay58}
Taylor, A.E.: Introduction to Functional Analysis.
\newblock Wiley, New York (1958)

\bibitem{Ter69}
Terzio{\~g}lu, T.: Die diametral {D}imension von lokalkonvexen {R}{\"a}umen.
\newblock Collect. Math. \textbf{20}, 49--99 (1969)

\bibitem{Ter73}
Terzio{\~g}lu, T.: Smooth sequence spaces and associated nuclearity.
\newblock Proc. Amer. Math. Soc. \textbf{37}(2), 497--504 (1973)

\bibitem{Ter75}
Terzio{\~g}lu, T.: Stability of smooth sequence spaces.
\newblock J. Reine Angew. Math. \textbf{276}, 184--189 (1975)

\end{thebibliography}
\bibliographystyle{spmpsci}
\end{document}